\documentclass[12pt]{amsart}
\usepackage{setspace}
\usepackage{amsmath,amsthm,amssymb,amsfonts}
\usepackage{hyperref}
\hypersetup{colorlinks=true}
\numberwithin{equation}{section}
\makeindex
\newcommand\xappa\kappa
\newcommand\yota\iota
\newcounter{consta}

\newcounter{constb}

\newcounter{constc}[section]

\newcommand{\vol}{\operatorname{vol}}

\newcommand{\R}{\mathbb{R}}

\newcommand{\Z}{\mathbb{Z}}

\DeclareFontFamily{OT1}{rsfs}{}
\DeclareFontShape{OT1}{rsfs}{n}{it}{<-> rsfs10}{}
\DeclareMathAlphabet{\mathscr}{OT1}{rsfs}{n}{it}
\swapnumbers
\newtheorem{thm}{Theorem}[section]
\newtheorem{lem}[thm]{Lemma}

\newtheorem{prop}[thm]{Proposition}
\newtheorem*{lem*}{Lemma}
\newtheorem*{thm*}{Theorem}
\newtheorem*{conj*}{Conjecture}
\newtheorem*{prop*}{Proposition}
\newtheorem{defn*}{Definition}

\newtheorem{ex}[thm]{Example}
\newtheorem{cor}[thm]{Corollary}

\theoremstyle{definition}
\newtheorem{defn}[thm]{Definition}
\theoremstyle{remark}
\newtheorem{rem}[thm]{Remark}
\newtheorem{obs}[thm]{Observation}

\newtheorem*{obs*}{Observation}
\newtheorem*{rem*}{Remark}
\theoremstyle{definition}\newtheorem*{acknowledgments}{Acknowledgments}

\begin{document}
	\title[Applications of joinings to sparse equidistribution]{Applications of joinings to sparse equidistribution problems}
	\author{Asaf Katz}
	\address{Department of Mathematics, University of Michigan, Ann Arbor, Michigan 48103, US}
	\email{asafk@umich.edu}
	
	\date{}
	
	\begin{abstract} We show how classification of joinings of two dynamical systems can be used  in some sparse equidistribution problems in homogeneous dynamics, and by using recent quantitative results about equidistribution theorems, one can deduce some quantitative estimates for such problems.
	\end{abstract}
	\maketitle
	\section{Introduction}
	In a landmark paper~\cite{furstenberg67}, H. Furstenberg introduced the notion of joinings of two dynamical systems and the concept of disjoint dynamical systems.
	Ever since, this property has played a major role in the field of dynamics, leading to many fundamental results.
	
	Given two ergodic measure-preserving systems $(X,\beta,\mu,T), (Y,\gamma,\nu,S)$, their set of joinings appear naturally when one considers the so-called ``Type II'' bi-linear sum $\mathcal{S}_{N,f_{1},f_{2}}(x,y)$ for two functions $f_{1} \in L^{\infty}(X,\mu), f_{2}\in~L^{\infty}(Y,\nu)$
	\begin{equation*}
		\mathcal{S}_{N,f_{1},f_{2}}(x,y)=\frac{1}{N}\sum_{n=0}^{N-1}f_{1}\left(T^{n}.x\right) f_{2}\left(S^{n}.y\right),
	\end{equation*}
	as by the ergodic theorem almost-surely we have 
	\begin{equation*}
		\mathcal{S}_{N,f_{1},f_{2}}(x,y) \to \int_{X\times Y}f_{1}\otimes f_{2}dJ_{x,y},
	\end{equation*}
	where $J_{x,y}$ is a joining of $\mu,\nu$ supported on $\overline{\left\{ \left(T^{n}.x,S^{n}.y\right)\right\}_{n\in \mathbb{N}}}$.
	Such sums appear naturally in equidistribution problems in the following situation - when one tries to estimate a given average of numbers $\frac{1}{N}\sum_{n=0}^{N-1}a_{n}$ by means of Van-der-Corput trick, which involves an extra smoothing of the sum (in a long range) and applying the Cauchy-Schwarz, then one needs to handle ``Type II sums'' of the form
	\begin{equation*}
		\frac{1}{(2H+1)^2}\sum_{\lvert h_{1}\rvert,\lvert h_{2}\rvert \leq H}\frac{1}{N}\sum_{n=0}^{N-1}a_{n+h_1}\cdot\overline{a_{n+h_{2}}},    
	\end{equation*}
	which may be interpreted by means of joining of the system with itself (and time-changes of).
	
	In this article we are interested in generalizing the following classical theorem:
	Set $\sigma_{R}$ to the probability one measure supported on a sphere of radius $R$ in $\mathbb{R}^{d}$.
	Let $\mathbb{T}^{d}=\mathbb{R}^{d}/\mathbb{Z}^{d}$ be the $n$-dimensional torus.
	We define the convolution operator associated to $\sigma_{R}$ as
	\begin{equation*}
		\sigma_{R}(f)(x_0)=\int_{v}f(v.x_0)d\sigma_{R}(v),
	\end{equation*}
	where the convolution is pushed forward by means of the natural $\mathbb{R}^{d}$ action on $\mathbb{T}^d$ by translations, namely $v.x_0=v+x_0 \mod \mathbb{Z}^{d}$, for any $f\in L^{\infty}(\mathbb{T}^{d})$.
	A classical theorem states that $\sigma_{R}(f)$ converges (as $R$ tends to infinity) to $\int_{\mathbb{T}^d}fdm$, where $m$ is the Haar probability measure defined on $\mathbb{T}^{d}$.
	
	We will show several generalizations of such a result in the settings of group actions on homogeneous spaces.
	
	Our first theorem, which is only qualitative in nature as it relies on Ratner's theorems, deals with the analogous situation of equidistribution of ''spherical averages`` along unipotent actions in homogeneous spaces associated with semi-simple Lie groups.
	Assume that $H\simeq \mathbb{R}^{n}$ acts by unipotent translations on a homogeneous space $X=G/\Gamma$ where $G$ is a connected semi-simple Lie group and $\Gamma$ is a lattice in $G$.
	Moreover, in the case where $\Gamma$ is non-uniform, we will assume that $\Gamma$ is arithmetic.
	Moreover, we assume that $H$ is a  \emph{horospherical group} $G_{a}^{+} = \left\{g\in G \mid a^n g a^{-n} \to e\text{ as }n\to\infty \right\}$, where $a$ is some $\mathbb{R}$-diagonalizable element.
	We define the compact subgroup $M\leq G$ as follows:
	\begin{equation*}
		M=Z^{\circ}_{K}(a),
	\end{equation*}
	where $K\leq G$ stands for the maximal compact subgroup of $G$.

	We define the $M$-spherical average to be
	\begin{equation*}
		\sigma_{R}(f)(x)=\int_{m\in M}f(m.a_{\log R}ua_{-\log R}.m^{-1}.x)dm,
	\end{equation*}
	for any given $u\in H\setminus\{e\}$, $f\in C_{c}(X)$.
	From now on, we will assume that $\dim(M)\geq 1$. Moreover by definition $M$ is connected.
	The main examples to keep in mind are $G=SL_{n+1}(\R)$ or $G=SO(n+1,1)(\R)$ with $M=SO(n)(\R)$.
	
	We note here that this choice of $M$ \emph{excludes} the case of $G=SL_{2}(\R)\times~\cdots~\times~SL_{2}(\R)$ which was discussed in~\cite{ubis2016effective}, as $M$ is trivial in that case.
	Furthermore, choosing a generic Cartan element $a=\text{diag}(e^{a_1},\ldots, e^{a_n})\in~SL_{n}(\R)$  with $a_1>~\ldots~>a_n$ would lead to $M$ being trivial. Hence our theorem will only apply to specific cases.
	\begin{thm}\label{thm:spherical-qual}
		In the previous notations, let $x\in X$ be an \emph{H}-generic point, namely $\overline{H.x}=~X$.
		Assume moreover that either $X$ is compact, or $\Gamma$ is arithmetic.
		Then for every $f\in~C_{c}(X)$ we have $\sigma_{R}(f)(x)~\to~\int_{X}fd\mu$ as $R\to \infty$.
	\end{thm}
	
	\begin{ex}
		Let $G=PSL_{d}(\mathbb{R})$ with $d\geq 3$.
		Let $\Gamma\leq G$ be a lattice.
		Define $H$ to be a minimal horospherical subgroup of $G$.
		Then $H\simeq~\mathbb{R}^{d-1}$, as can be seen by choosing for example $a=\text{diag}\left(e^{1/d},\ldots,e^{1/d},e^{-(d-1)/d}\right)$.
		$M$ in this case is isomorphic to $SO(d-1)(\mathbb{R})$, and the equation of conjugation by $M$ over $H$ amounts to rotation in the sense of multiplication by the $M$ matrix in $\mathbb{R}^{d-1}$.
		The ''ball averages`` 
		\begin{equation*}
			\beta_{R}(f)(x)=\frac{1}{R}\int_{r=0}^{R}\sigma_{r}f(x)dr,
		\end{equation*}
		are known to equidistribute (quantitatively, subject to diophantine conditions over the base point) for any $H$-generic point.
		Our theorem shows that the sparser averages $\sigma_{r}(f)(x)$ equidistribute (qualitatively) by themselves as $r$ grows.
	\end{ex}
	
	Our next theorem deals with quantitative version of the above theorem.
	Unfortunately, in the semisimple case, we cannot recover a full quantitative version of the spherical equidistribution theorem, but only for a thicker kind of averages.
	\begin{thm}\label{thm:annuli-equi}
		Let $G$ be a linear semi-simple Lie group and let $\Gamma\leq~G$ be a \emph{lattice} in $G$.
		Denote $X=G/\Gamma$ the associated homogeneous space and let $m$ denote the probability measure defined on $G/\Gamma$ which is induced by a Haar measure on $G$.
		Let $H\leq G$ be an abelian subgroup.
		Assume that $x\in X$ is an $H$-generic point, namely $\overline{H.x}=X$ which has an equidistribution rate of $\gamma_{\text{E}}>0$ with respect to function of bounded Sobolev norm of order $K$ (see Definition~\ref{def:quantitative-equi-rate}.
		Then there exists an exponent $\omega_{\text{critical}}<1$ such that for all other exponents $\omega_{\text{critical}}<\omega\leq 1$ we have the following average over the annulus of inner radius $R/2$ and outer radius $R/2+R^{\omega}$, namely
		\begin{equation*}
			A_{R,\omega}(f)(x) = \frac{1}{R^{\omega}}\int_{t=0}^{R^{\omega}}\sigma_{R/2+t}(f)(x)dt,
		\end{equation*}
		equidistributes with a rate $\delta=\delta(\gamma_{\text{E}},\omega)>0$, namely
		\begin{equation*}
			\left\lvert A_{R,\omega}(f)(x)-\int_{X}fdm \right\rvert \ll_{f} R^{-\delta}.
		\end{equation*}
	\end{thm}
	The proof relies on sieving the larger ``ball`` average 
	$$\beta_{R}(f)(x)=\frac{1}{R}\int_{t=0}^{R}\sigma_{t}(f)(x)dt$$ using quantitative disjointness result of this $\mathbb{R}^{d}$-action from the Kroncker system.
	Another application of the quantitative disjointness result, similar in spirit to Venkatesh's, regarding singular averages associated with Bochner-Riesz means appears in Theorem~\ref{thm:Bocher-Riesz}.
	We remark here that the proof only requires a quantitative equidistribution theorem for the $H$-orbit, $H.x$, which by current results known only (in full generality) for the case of horospherical actions.

	Our third theorem deals with quantitative dilations of ''spheres`` in nilmanifolds.
	\begin{thm}\label{thm:nil-equi}
		Let $X=~N/\Lambda$, where $N$ is a nilpotent connected Lie group and $\Lambda\leq N$ is a \emph{lattice}, admitting an $\mathbb{R}^{d}$-action by translations.
		Assume moreover that this action is minimal.
		We define the sphere average of radius $R>0$ around $x_0\in X$ as
		\begin{equation*}
			\sigma_{R}f(x_0)=\frac{1}{\vol(\partial B_{R})}\int_{\lVert \overline{v}\rVert=R}f\left(u_{v}.x_0\right)d\overline{v},
		\end{equation*}
		for any $f\in C(X)$.
		Then for any point $x_0$ which satisfy effective equidistribution theorem, the sphere averages equidistribute towards $\int_{X}f(x)dm$ as $R\to \infty$ in a quantitative manner.
	\end{thm}
	See Theorem~\ref{thm:green-tao}, Definition~\ref{def:quantitatie-equi-rate-nil}, Remark~\ref{rem:effective-equi} and Equation~\eqref{eq:effective-equi-nil-estimate} for discussion regarding the quantitative estimates achieved.
	We note that in the case where the nilmanifold $N/\Lambda$ is the $n$-dimensional torus $\mathbb{T}^n=\mathbb{R}^n/\mathbb{Z}^n$, our theorem is merely the well-known estimate for decay of spherical Bessel functions (c.f. \cite{burton}), and as such one may interpret the above estimate as a decay estimate for ``nilpotent Bessel functions''.
	In recent work~\cite{brynakrashah} Kra, Shah and Sun have considered equidistribution of dilated curves in nilmanifolds, it is highly likely that an extension of our methods will be applicable there as well.
	
	\section*{Organization of the paper}
	In \S2, we prove Theorem~\ref{thm:spherical-qual}, using the linearization technique of Dani-Margulis.
	In \S3, we prove Theorem~\ref{thm:annuli-equi}, using quantitative disjointness theorem between Kronecker systems and horospherical flows.
	In \S4, we prove Theorem~\ref{thm:nil-equi} by means of analysis of nilcharacters over joinings on the related nilflows.
	
	\begin{acknowledgments}
		The author is indebted to Elon Lindenstrauss for introducing the problem of spherical equidistribution to him and many discussions about the subject of quantitative equidistribution.
		The author also wishes to thank Alex Eskin for fruitful discussions and to Kevin Hughes and Boaz Haberman for useful advice about harmonic analysis.
		The author thanks the referees for carefully the paper and providing useful comments.
	\end{acknowledgments}
	
	\section{Equidistribution of spheres in horospheres}
	In this section, we provide a \emph{qualitative proof} for the equidistribution of large spheres in horospheres. We note here that the proof is not quantitative in its nature, as it relies on Ratner's theorem and the linearization technique of Dani-Margulis.
	
	Let $G$ be a fixed semisimple real Lie group.
	Let $\Gamma\leq G$ be a lattice.
	Denote $X=G/\Gamma$ and denote by $\mu$ the corresponding probability measure on $X$ associated with the Haar measure of $G$.
	We fix an $\mathbb{R}$-diagonalizable $a\in G$ element such that $H=\left\{ g\in G \mid a^{n}ga^{-n} \xrightarrow{n\to -\infty} e \right\}$.
	Moreover, assume $H\simeq R^{d}$.
	We denote by $K$ the maximal compact compact subgroup of $G$.
	We denote by $M=Z_{K}(a)\leq K$.
	\begin{obs}
		The subgroup $M$ normalizes $H$.
	\end{obs}
	\begin{proof}
		For any $h\in H, m\in M$ we have
		\begin{equation*}
			a^{n}(mhm^{-1})a^{-n} = m(a^{n}ha^{-n})m^{-1}\to mem^{-1}=e,
		\end{equation*}
		as $H$ is horospherical with respect to the group element $a$.
	\end{proof}
	A point $x\in X$ is called \emph{$H$-generic} if $\overline{H.x} = X$, we note that by the ergodic theorem, $\mu$-a.e. point $x\in X$ is $H$-generic.
	We define the following spherical averaging operator $\sigma_{R}f(x)$ as follows:
	\begin{equation}\label{eq:M-avg}
		\begin{split}
			\sigma_{R}(f)(x) &= \int_{\lVert \hat{v} \rVert=1}f(u_{R\cdot \hat{v}}.x)d\sigma_{\hat{v}} \\
			&= \int_{m\in M}f(m.u_{R\cdot \hat{v}}.m^{-1}.x)dm_{M},
		\end{split}
	\end{equation}
	for any normalized vector $\hat{v}$ in $\mathbb{R}^{d}$, where $dm_{M}$ stands for the normalized Haar probability measure on the compact subgroup $M\leq G$.

	\subsection{Non-divergence of dilates of $M$-averages}
	Our first theorem deals with escape of mass under this averaging operator.
	We will use the terminology and notations of the profound non-divergence theorem proven by Kleinbock and Margulis in~\cite{KleinbockMargulis}.
	For simplicity of presentation, we will consider the case of $X=SL_{n+1}(\R)/SL_{n+1}(\Z)$. The case of any other arithmetic quotient follows easily from that case.
 In this settings, the horospherical subgroup is isomorphic to $R^{n}$, with the choice of a diagonal element $a=\text{diag}(e^{t/2},e^{-t/2\cdot n},\ldots, e^{-t/2\cdot n})$.
	
	\begin{defn}
		A \emph{primitive subgroup} $\Delta\leq \mathbb{Z}^k$ is a subgroup such $\mathbb{R}\otimes \Delta \cap \Z^k = \Delta$.
	\end{defn}
	We may identify each such primitive subgroup $\Delta$ with a vector in $\bigwedge^{\text{rank}(\Delta)}\mathbb{Z}^{n+1} \leq \bigwedge^{\text{rank}(\Delta)}\mathbb{R}^{n+1}$ by means of an exterior product of any basis.
	This gives rise to a norm of a primitive subgroup, as the norm of any such basis realization in the exterior product.
	Let $\Delta$ be a primitive $\mathbb{Z}^{n}$-subgroup, define the function $\Phi_{x,\Delta}(v):~\mathbb{R}^{n}\to~\mathbb{R}$ as 
	\begin{equation*}
		\Phi_{x,\Delta}(v) = \lVert u_{v}.x.\Delta \rVert, 
	\end{equation*}
	and the assorted spherical function $\Phi_{x,\Delta,R}(v):R\cdot S^{n-1}\to\mathbb{R}$ as $$\Phi_{x,\Delta,R}=~\Phi_{x,\Delta}\mid_{R\cdot S^{n-1}}.$$
	
	We will need the following definition.
	\begin{defn}[(C,$\alpha$)-good function]
		Let $B\subset\mathbb{R}^{n}$ be a convex subset. A function $f:B \to \mathbb{R}$ is called \emph{(C,$\alpha$)-good} for some $C,\alpha>0$ if for any $B'\subset B$ open and $\varepsilon>0$ one have
		\begin{equation*}
			meas\left\{ x\in B' \mid \lvert f(x) \rvert < \varepsilon \right\} \leq C\cdot \left(\frac{\varepsilon}{\sup_{x\in B'}\lvert f(x) \rvert} \right)^{\alpha} \cdot meas(B').
		\end{equation*}
	\end{defn}
	
	\begin{lem}[The functions $\{\Phi_{x,\Delta,R}\}$ are $(C,\alpha)$-good]\label{lem:c,a-good}
		The maps $\left\{\Phi_{x,\Delta,R} \right\}_{R}$ are $(C,\alpha)$-good maps.
	\end{lem}
	This follows immediately from \cite[Theorem~$2.7$, Proposition~$2.8$]{aka}.
	Note that in view of \cite[Proposition~$3.4$]{KleinbockMargulis}, it is essentially enough to show that the first $\ell$ derivatives of the map $v\mapsto u_{v}.x\simeq \mathbb{R}^{n}$ span $\mathbb{R}^n$. As $S^{n-1}$ is embedded as a regular surface in $\mathbb{R}^{n}$ (except the poles), its derivatives span the tangent space.
    It is an easy computation to show that the second derivatives are not contained in the tangent plane as well.
    alternatively, as $R\cdot S^{n-1}$ is an analytic connected subvariety of $\mathbb{R}^{n}$, if its derivative not span $\mathbb{R}^{n}$, then it is contained in a proper affine subspace of co-dimension at least $1$, which clearly does not hold.

    We note that with our choice of $H\simeq \mathbb{R}^{n}$. the action of the horospherical subgroup $H$ over a vector $v\in \bigwedge\mathbb{R}^{n+1}$ was calculated in~\cite[\S5]{KleinbockMargulis}. We recall the computation below.
    Given a subset of indices $I=\left\{i_{1}<i_{2}<\ldots < i_{k} \right\}$, we define the vector $e_{I}$ as $e_{I} = \bigwedge_{j=1}^{k}e_{i_{j}}$, with $e_{\phi}=1$. Vectors of the form $\left\{ e_{I} \right\}_{I\subset \left\{ 0,\ldots, n\right\}}$ span the exterior algebra $\bigwedge \mathbb{R}^{n+1}$.
    We define a norm over $\bigwedge \mathbb{R}^{n}$ by
    \begin{equation*}
        \left\lVert v\right\rVert = \left\lVert \sum_{I} \alpha_{I}v_{I} \right\rVert = \max_{I} \left\lvert \alpha_{I}\right\rvert.
    \end{equation*}
    We note that for $y=(y_{1},\ldots, y_{n})\in \mathbb{R}^{n}$ we have:
    \begin{equation*}
        u_{y}.e_{I} = \begin{cases}
        e_{I}  &0\in I \\
        e_{I}+\sum_{i\notin I}(-1)^{\ell(I,i)} y_{i}\cdot e_{I\cup\{i\}\setminus\{0\}} &0\notin I,
    \end{cases}
    \end{equation*}
    where $\ell(I,i)$ is the number of elements in $I$ strictly between $0$ and $i$.
    As a result, we get that for $v=\sum_{I}\alpha_{I}v_{I}$
    \begin{equation}\label{eq:wedge-action}
        u_{y}.v = \sum_{0\notin I}\alpha_{I}v_{I} + \sum_{0\in I}\left(\alpha_{I}+\sum_{i\notin I}(-1)^{\ell(I,i)} y_{i}\cdot \alpha_{I\cup\{i\}\setminus\{0\}} \right)\cdot v_{I}.
    \end{equation}

    \begin{lem}
        For any $\Delta\in Prim(\mathbb{Z}^{n+1})$, for all $R\geq 0$, we have that
        \begin{equation*}
            \left\lVert \Phi_{x,\Delta}(y) \right\rVert = \max_{y\in R\cdot S^{n-1}}\left\lVert u_{y}\cdot x\Delta \right\rVert \geq \left\lVert x\Delta \right\rVert.
        \end{equation*}
    \end{lem}
    \begin{proof}
        We denote $v=x\Delta$ to be the vector representing $x\Delta$ in $\bigwedge \mathbb{R}^{n+1}$, so $v=\sum_{I} \alpha_{I}e_{I}$ for $\alpha_{I}=\alpha_{I}(x\Delta)$.
        It follows from~\eqref{eq:wedge-action} that
        \begin{equation*}
            \left\lVert u_{y}.v\right\rVert \geq \max\left\{\left(\left\lvert \alpha_{I}\right\rvert\right)_{\text{ }0\notin I}, \left(\left\lvert \alpha_{I}+\sum_{i\notin I}(-1)^{\ell(I,i)} y_{i}\cdot \alpha_{I\cup\{i\}\setminus\{0\}} \right\rvert\right)_{\text{ }0\in I} \right\}.
        \end{equation*}
        We consider two cases. If $\left\lVert v\right\rVert = \left\lvert \alpha_{I}\right\rvert$ for $I$ such that $0\notin I$, then then above inequality yields $\left\lVert u_{y}.v\right\rVert \geq \left\lvert \alpha_{I}\right\rvert = \left\lVert v\right\rVert$ for all $y\in\mathbb{R}^{n}$.
        In the other case, $\left\lVert v\right\rVert = \left\lvert \alpha_{I}\right\rvert$ for $I$ such that $0\in I$.
        Consider such a subset $I$, pick $y\in R\cdot S^{n-1}$ to be $y=\pm(0,\ldots, R,0,\ldots, 0)$, such that $R$ is in the $i$'th entry for some index $i\in I$.
        We get that for this choice of $y$, with the appropriate sign with respect to $(-1)^{\ell(I,i)}\cdot \alpha_{I\cup\{i\}\setminus\{0\}}$
        \begin{equation*}
        \begin{split}
            \left\lvert \alpha_{I}+\sum_{i\notin I}\pm y_{i}\cdot \alpha_{I\cup\{i\}\setminus\{0\}} \right\rvert &= \left\lvert\alpha_{I} \pm R\cdot \alpha_{I\cup\{i\}\setminus\{0\}}\right\rvert \\
            &= \left\lvert\alpha_{I}\right\rvert + R\cdot \left\lvert\alpha_{I\cup\{i\}\setminus\{0\}}\right\rvert \\
            &\geq \left\lvert \alpha_{I}\right\rvert,
        \end{split}
        \end{equation*}
        which proves the claim in this case.
    \end{proof}

    A closer inspection of the above proof actually yields the following general theorem:
    \begin{thm}
        Assume that $f:U\to \mathbb{R}^{n}$ is a map such that $f(U)$ contains \emph{antipodal points} on all the axis of $\mathbb{R}^{n}$, then for any $v\in \bigwedge \mathbb{R}^{n+1}$ we have
        \begin{equation*}
            \sup_{x\in U}\left\lVert u_{f(x)}.v \right\rVert \geq \left\lVert v \right\rVert.
        \end{equation*}
    \end{thm}

    \begin{cor}[The functions $\{\Phi_{x,\Delta,R}\}$ are $\rho$-big]\label{cor:sup-rho}
        For every $\Delta \in Prim(\mathbb{Z}^{n})$, for any $x\in SL_{n}(\mathbb{R})$ and $R\geq 0$ we have that
        \begin{equation*}
            \max\left\lVert \Phi_{x,\Delta}(v) \right\rVert = \max_{y\in R\cdot S^{n-1}} \left\lVert u_{y}.x.\Delta \right\rVert \geq \left\lVert x.\Delta\right\rVert.
        \end{equation*}
        In particular, if we set $\rho = \min_{\Delta \in Prim(\mathbb{Z}^{n}}\left\lVert x.\Delta\right\rVert$, we have that
        \begin{equation*}
            \max_{y\in R\cdot S^{n-1}} \left\lVert u_{y}.x.\Delta \right\rVert \geq \rho.
        \end{equation*}
    \end{cor}
    The first assertion follows at once from the Lemma above. The second assertion follows from the first by noticing that $\bigwedge^{k}\mathbb{Z}^{n} \leq \bigwedge^{k} \mathbb{R}^{n}$ is a discrete subset.

    Recall that for all $\epsilon>0$ the set $X_{>\epsilon} = \left\{ x \in SL_{n}(\mathbb{R})/SL_{n}(\mathbb{Z}) \mid \min_{v\in x\mathbb{Z}^{n}\setminus{0}} \left\lVert v\right\rVert > \epsilon\right\}$ is \emph{precompact}, by Mahler's compactness criterion. 

    Combining Corollary~\ref{cor:sup-rho} and Lemma~\ref{lem:c,a-good} with the celebrated non-divergence theorem of Kleinbock-Margulis~\cite[Theorems~$5.2-5.3$]{KleinbockMargulis} gives
    \begin{thm}
        There exists $D>0$ such that for any $x\in SL_{n}(\mathbb{R})/SL_{n}(\mathbb{Z})$, there exists $\rho=\rho(x)>0$ such that for all $R\geq 0$ and $\epsilon>0$
        \begin{equation*}
            Leb\left\{ y\in R\cdot S^{n-1} \mid u_{y}.x \notin X_{> \epsilon} \right\} \ll \left(\frac{\epsilon}{\rho}\right)^{D},
        \end{equation*}
        where $Leb$ is the normalized Lebesgue measure on the sphere.
    \end{thm}
    Using the identification of the $M$ action over the horosphere and the sphere, we get the equivalent estimate
    \begin{equation}\label{eq:non-div-estimate}
        m_{M}\left\{ m\in M \mid m.u_{R\cdot e_{1}}.m^{-1}.x \notin X_{>\epsilon} \right\} \ll \left(\frac{\epsilon}{\rho}\right)^{D}.
    \end{equation}

    We end up the discussion with an alternative proof of the $\rho$-bigness property for the embedded manifolds in the horosphere, which applies to slightly more general class of maps into horospheres.
    \begin{thm}
        Let $f:U\to \mathbb{R}^{n}\simeq H$ be a continuous map into a horospherical subgroup, such that for all $f_{i}$, we have $\int_{U}f_{i}(x)dx = 0$.
        Then for any $v\in \bigwedge\mathbb{R}^{n+1}$ we have
        \begin{equation*}
            \sup_{x\in U}\left\lVert u_{f(x)}.v\right\rVert \gg_{U} \left\lVert v\right\rVert. 
        \end{equation*}
    \end{thm}
    \begin{proof}
        Using the explicit description of the action by Kleinbock-Margulis as in~\eqref{eq:wedge-action} we get again that for all $x\in U$:
        \begin{equation*}
            \left\lVert u_{f(x)}.v\right\rVert = \max\left\{\left(\left\lvert \alpha_{I}\right\rvert\right)_{\text{ }0\notin I}, \left(\left\lvert \alpha_{I}+\sum_{i\notin I}(-1)^{\ell(I,i)} f_{i}(x)\cdot \alpha_{I\cup\{i\}\setminus\{0\}} \right\rvert\right)_{\text{ }0\in I} \right\}.
        \end{equation*}
        Pick an index set $I$ so that $\left\lVert v\right\rVert = \left\lvert \alpha_{I}\right\rvert$ and we consider two cases, whether $0\notin I$ or $0\in I$.
        In the first case, where $\left\lVert u_{f(x)}.v\right\rVert = \left\lvert \alpha_{I}\right\rvert$, the claim is obvious.
        In the second case, we see that
        $\left\lVert u_{f(x)}.v\right\rVert = \left\lvert \alpha_{I}+\sum_{i\notin I}(-1)^{\ell(I,i)} f_{i}(x)\cdot \alpha_{I\cup\{i\}\setminus\{0\}} \right\rvert$.
        As for all $i$, $\int_{U}f_{i}(x)dx = 0$, we get that
        \begin{equation*}
            \int_{U} \alpha_{I}+\sum_{i\notin I}(-1)^{\ell(I,i)} f_{i}(x)\cdot \alpha_{I\cup\{i\}\setminus\{0\}}dx = \alpha_{I}.
        \end{equation*}
        Therefore
        \begin{equation*}
        \begin{split}
            \sup_{U}\left\lVert u_{f(x)}.v \right\rVert &\geq \frac{1}{Leb(U)}\int_{U}\left\lVert u_{f(x)}.v\right\rVert dx \\
            &\geq \frac{1}{Leb(U)}\int_{U}\left\lvert \alpha_{I}+\sum_{i\notin I}(-1)^{\ell(I,i)} f_{i}(x)\cdot \alpha_{I\cup\{i\}\setminus\{0\}} \right\rvert dx \\
            &\geq \frac{1}{Leb(U)}\left\lvert \int_{U} \alpha_{I}+\sum_{i\notin I}(-1)^{\ell(I,i)} f_{i}(x)\cdot \alpha_{I\cup\{i\}\setminus\{0\}} dx\right\rvert \\
            &\geq \frac{1}{Leb(U)}\left\lvert \alpha_{I}\right\rvert
        \end{split}
        \end{equation*}
        \begin{equation*}
        \begin{split}
            \int_{U}\left\lvert \alpha_{I}+\sum_{i\notin I}(-1)^{\ell(I,i)} f_{i}(x)\cdot \alpha_{I\cup\{i\}\setminus\{0\}} \right\rvert dx 
            &\geq \left\lvert \int_{U} \alpha_{I}+\sum_{i\notin I}(-1)^{\ell(I,i)} f_{i}(x)\cdot \alpha_{I\cup\{i\}\setminus\{0\}} dx\right\rvert \\
            &\geq Leb(U)\left\lvert \alpha_{I}\right\rvert
        \end{split}
        \end{equation*}
        In particular, there exists a point $x\in U$ for which $\left\lVert u_{f(x)}.v\right\rVert \gg_{U} \left\lvert \alpha_{i}\right\rvert$, and the claim follows.
    \end{proof}

    \begin{ex}
        Consider the map $f(x)=(x,x-cos(x))$ for $x\in [-\pi,\pi]$.
        We have $\int_{-\pi}^{\pi}xdx=0, \int_{-\pi}^{\pi}\cos(x)dx = 0$, hence the theorem above applies for the curve $u_{f(x)}$.
        Another example would be $(\cos(x),\cos(2x))$ where say $x\in[-\pi,\pi]$.
    \end{ex}

    We note here that one can use the Theorem to verify $\rho$-bigness for a family of translates. Also in the case of dilations of maps over horospheres, the claim holds as well. One notes that all the various dilates are parameteried by a \emph{common} set $U$, and all the coordinate functions of the said dilates will still be of vanishing integral.
    
	\begin{rem}
		The case of general arithmetic lattice follows from the proof we presented by analyzing the adjoint representation and exterior powers of. Splitting the representation with respect to weights induced by $a$ (a defining rational split Cartan element of the horosphere), one can see that the elements of the horospherical subgroup act as raising operators (to the positive weight space). In particular, considering the action on any $Lie(G)_{\mathbb{Z}}$-rational vector $v$, we see the the contracting part of the vector stays the same, the neutral part of the vector has the contracting part added to it and the expanding part gets the neutral part added to it. Using the same logic, by choosing the appropriate component like before, we can arrange this component to grow, at least for one point in $R\cdot S^{n-1}$. The $(C,\alpha)$-good property holds in general.
	\end{rem}
	
	\subsection{Equidistribution of dilates of $M$-averages}
	We start our analysis of the spherical average by the following approximation argument:
	\begin{prop}\label{prop:approx}
		There exists some $\beta>0$ such that for any bounded Lipschitz function $f:X\to\mathbb{R}$ and any $R\gg 0$ we have
		\begin{equation}\label{eq:extra-avg}
			\begin{split}
				&\left\lvert \int_{m\in M}f(m.u_{R\cdot \hat{v}}.m^{-1}.x)dm_{M}-\int_{m\in M}\frac{1}{\vol(B_{R^{\beta}}^{\perp})}\int_{v'\in B_{R^{\beta}}^{\perp}}f(m.u._{v'}.u_{R\cdot \hat{v}}.m^{-1}.x)dv'dm_{M} \right\rvert \\ & \ \ll_{f} R^{2\beta-1},
			\end{split}
		\end{equation}
		where $B_{R^{\beta}}^{\perp}$ stand for the intersection of $B_{R^{\beta}}$ with the subspace $\left\{x\in \mathbb{R}^{n} \mid \left<x,\hat{v}\right>=0 \right\}$.
	\end{prop}
	Essentially the Proposition shows we may approximate the average along a sphere by a combination of translated lower-dimensional averages along its tangent planes, due to curvature of the sphere.
    \begin{proof}
    Consider a cap of width $2R^{\beta}$ in $R\cdot S^{n-1}$.
    The maximal difference between the tangent plane to the actual sphere over this cap is equal to $R-\sqrt{R^{2}-R^{2\beta}} = \frac{R^{2\beta}}{R+\sqrt{R^{2}-R^{2\beta}}} \leq R^{2\beta-1}$.
    Using a Lipschitz estimate for the difference of the normalized evaluation of the function over the cap of sphere and over the tangent, we see that over the cap, the difference is $O(R^{2\beta}-1)$.
    As the Lebesgue measure over the sphere is invariant under rotations, the result follows.
    \end{proof}
	
    \begin{lem}
	Assume that $x\in X$ is $H$-generic point.
	There exists a subgroup $H'\leq H$ such that $H'\simeq \mathbb{R}^{d-1}$ so that $x$ is $H'$-generic.
    \end{lem}
    Our proof is closely related to the proof of~\cite[Corollary~$1.2$]{shah-poly}.
    \begin{proof}
		Suppose that for every 
		$\hat{v}\in S^{d-1}$ we have that $$\frac{1}{\vol(B_{R^{\beta}}^{\perp})}\int_{v\in B_{R}^{\perp}}f(u_{v}.x)dv \not\to ~\int_{X}fd\mu.$$
		By Ratner's equidistribution theorem, we have that $$\frac{1}{\vol(B_{R^{\beta}}^{\perp})}\int_{v\in B_{R}^{\perp}}f(u_{v}.x)dv \to~ \int_{X}fd\mu_{\hat{v}}$$ for some algebraic probability measure $\mu_{\hat{v}}$.
		Let $\mathcal{H}$ denote the collection of all closed connected subgroups $Q\leq G$ such that $Q\cap \Gamma \leq Q$ is a lattice and $Q_{\text{uni}}$ acts ergodically on $Q/Q\cap \Gamma$ with respect to the $Q$-invariant probability measure, where $Q_{\text{uni}}\leq Q$ denotes the subgroup generated by all one-parameter unipotent subgroups of $Q$.
		Ratner's countability theorem~\cite[Theorem~$1.1$]{ratner91} asserts the $\mathcal{H}$ is countable.
		Note that $\text{Stab}\left(\mu_{\hat{v}}\right) \in\mathcal{H}$.
		Denote by $\mathcal{H}_{H}$ to be the subset of $\mathcal{H}$ which is composed of subgroups $Q\leq H$ so that $Q\in \mathcal{H}$.
		For each $Q\in\mathcal{H}_{H}$ we define the $Q$-tube (with respect to $\text{Stab}\left(\mu_{\hat{v}}\right)$) (c.f. \cite[Section~$3$]{dm93},\cite[Theorem~$2.4$]{mozes-shah}) as
		\begin{equation}\label{eq:tube-def}
			T_{Q}=\pi\left(N(Q,\text{Stab}\left(\mu_{\hat{v}}\right))\setminus S(Q,\text{Stab}\left(\mu_{\hat{v}}\right)) \right),
		\end{equation}
		where $\pi:G\to G/\Gamma$ is the natural projection map, 
		\begin{equation*}
			N(Q,\text{Stab}\left(\mu_{\hat{v}}\right))=\left\{g\in G \mid \text{Stab}\left(\mu_{\hat{v}}\right)\subset gQg^{-1} \right\},
		\end{equation*}
		and
		\begin{equation*}
			S(Q,\text{Stab}\left(\mu_{\hat{v}}\right))=\cup_{Q'\in\mathcal{H}, Q'\subsetneqq Q}N(Q',\text{Stab}\left(\mu_{\hat{v}}\right)).
		\end{equation*}
		Each such $Q$-tube is an analytic subvariety.
		Therefore, we must have for any $Q\in\mathcal{H}$ which is not equal to $G$ that $m\left\{\hat{v}\in S^{d-1} \mid \text{Stab}(\mu_{\hat{v}})=Q\right\}=~0$.
		Otherwise, we would have that for every $\hat{v}\in S^{d-1}$ we have $\text{Stab}(\mu_{\hat{v}})=~Q$ for some $Q\in\mathcal{H}$, as $\text{Stab}(\mu_{\hat{v}})$ is a connected algebraic subgroup.
		Therefore we have that 
		\begin{equation*}
			\begin{split}
				m\left\{\hat{v} \in S^{d-1} \mid Stab(\mu_{\hat{v}})=G \right\} &= 1-\sum_{Q\in \mathcal{H}\setminus\{G\}}m\left\{\hat{v}\in S^{d-1} \mid Stab(\mu_{\hat{v}})=Q\right\}\\
				&=1.
			\end{split}
		\end{equation*}
		Therefore any random choice of $\hat{v}\in S^{d-1}$ satisfies the conclusion.
	\end{proof}
	
	The following theorem of Dani-Margulis provides uniformity over the convergence speed of the set of $H'$-generic points.
	The same theorem is used in~\cite[Theorems~$4.1,4.3,4.4$]{eskin-margulis-mozes}
	\begin{thm}[\cite{dm93},\cite{eskin-margulis-mozes}]\label{thm:DM-uni}
		Assume the above mentioned settings for $G,\Gamma,H'$, and fix a compact set $D\subset X$. For a fixed bounded function $F:X\to \mathbb{R}$ and any $\varepsilon>0$ there exists finitely many points $x_{1},\ldots,x_{N}\in X$ so that the orbits $\left\{H'.x_{i}\right\}_{i=1}^{N}$ are closed, and for any compact subset $$C\subset D\setminus \bigcup_{i=1}^{N}H'.x_{i}$$ we have some $R_{0}=R_{0}(C)>0$ such that for any $R>R_{0}$ and any $y\in C$ we have
		\begin{equation*}
			\left\lvert \int_{v\in B_{R}^{\perp}}F(u_{v}.y)dv - \int_{X}Fd\mu \right\rvert < \varepsilon.
		\end{equation*}
	\end{thm}

	Let $M$ be a compact Lie group, $\rho$ a unitary representation of $M$.
	Let $\hat{M}$ denote the set of equivalence classes of finite dimensional irreducible representations of $M$.
	\begin{defn}
		We say that $f:X\to\mathbb{C}$ has $M$-type equal to $\delta\in~\hat{M}$ if the $M$-representation given by $m\to m.f \in \langle M.f \rangle \leq L^{2}(X)$ is isomorphic to $\delta$.
	\end{defn}
	We say the $f$ is $M$-fixed if its $M$-type corresponds to the trivial representation.
	Namely we get that $f(m.x)=f(x)$ for any $m\in M$, $x\in X$.
	We say that $f$ is $M$-finite if $\dim \langle M.f \rangle<\infty$.
	In general, for admissible representations, we have that the $M$ representation coming from $M.f$ breaks down to a countable sum of finite-dimensional irreducible $M$-representations, indexed by the various $M$-types $\delta$, each appearing with finite multiplicity, by the Peter-Weyl theorem.
	Assume that $f$ has $M$-type equal to $\delta$.
	There are continuous functions $a_{i}:M\to\mathbb{C}$, for $i=1,\ldots,n$ where $n=\dim \delta$, such that 
	\begin{equation}\label{eq:M-type-func}
		m.f=\sum_{i=1}^{n}a_{i}(m)\cdot f_{i}
	\end{equation}
	for some basis $\left\{f_{i}\right\}_{i=1}^{n}$ of $\langle M.f \rangle$, with $\sum_{i=1}^{n}\lvert a_{i} \rvert^{2}=\lVert f \rVert^{2}_{2}$.
	
	Assume now that $f$ has a fixed $M$-type.
	Per~\eqref{eq:M-type-func}, one may write the spherical average as follows:
	\begin{equation}\label{eq:m-type-equi}
		\begin{split}
			&\int_{m\in M}\frac{1}{\vol(B_{R^{\beta}}^{\perp})}\int_{v\in B_{R^\gamma}^\perp}f(m.u_{v}.u_{R\cdot \hat{v}}.m^{-1}.x)dvdm_{M} \\
			&\ \ = \sum_{i=1}^{n}\int_{m\in M}a_{i}(m)\cdot\frac{1}{\vol(B_{R^{\beta}}^{\perp})}\int_{v\in B_{R^{\gamma}}^\perp} f_{i}(u_{v}.u_{R\cdot \hat{v}}.m^{-1}.x)dvdm_{M}.
		\end{split}
	\end{equation}
	Hence by showing equidistribution of the inner average \\  $\frac{1}{\vol(B_{R^{\beta}}^{\perp})}\int_{v\in B_{R^{\gamma}}^\perp} f_{i}(u_{v}.u_{R\cdot \hat{v}}.m^{-1}.x)dv$ we may conclude equidistribution for function of fixed $M$-type.
	This idea goes back to Eskin-Margulis-Mozes~\cite[Equation~$(4.2)$]{eskin-margulis-mozes}.
	
	The following lemma allows us to bootstrap such equidistribution result into general functions, using the density of linear combinations of functions with finite $M$-type.
	\begin{lem}
		Denote by $C_{c}^{M}(X)$ to be the space of functions over $X$ which are continuous and of compact support having a finite $M$-type, then this space is dense in $C_{c}(X)$.
	\end{lem}
	\begin{proof}
		Let $V$ be a Banach space representation for $G$.
		It is known that the space of smooth vectors, $V^{\infty}\subset V$ is dense in V.
		So we may restrict our attention to such a smooth vector $f\in V^{\infty}$.
		We associate with $f$ the following $M$-representation: $\overline{\left< M.f \right>}\leq V$, which we may assume being unitary $M$-representation by Weyl's unitary trick.
		Using the Peter-Weyl theorem, we may write this representation as follows 
		\begin{equation*}
			\overline{\left< M.f \right>}=\oplus V_{i},
		\end{equation*}
		where $V_{i}$ is an irreducible unitary $M$-representation and the sum is countable, with finite-multiplicity for each $M$-representation.
		For each such $M$-representation $\pi$ involved in $\oplus V_{i}$, we define the projection of $v\in \overline{\left< M.f \right>}$ on $\pi$ as
		\begin{equation*}
			P(\pi).v=d(\pi)^{2}\int_{m\in M}\overline{\sigma_{\pi}(m)}\cdot (\pi(m).v) dm_{M},
		\end{equation*}
		where $d(\pi)$ is the \emph{formal dimension} of $\pi$ and $\sigma_{\pi}(m)$ stands for the character of $\pi$.
		Correspondingly, we may define the $M$-Fourier expansion of $v\in \overline{\left< M.f \right>}$ as
		\begin{equation*}
			F_{M}(v)=\sum_{\pi_{i}} P(\pi_{i}).v,
		\end{equation*}
		where $\pi_{i}$ is the $M$-representation associated to $V_{i}\leq \overline{\left< M.f \right>}$.
		By a theorem of Harish-Chandra\cite[Section~$\S 4.4.2$]{warner}, the Fourier expansion converges absolutely to $v$ for any differentiable vector $v\in V^{\infty}$.
	\end{proof}
	
	Now we are in position to prove the spherical equidistribution theorem.
	We will make use of the following proposition, observed by Eskin-Margulis-Mozes.
	\begin{lem}[\cite{eskin-margulis-mozes}, Lemma~$4.2$]\label{lem:eskin-margulis-mozes}
		Let $\mathcal{K}$ be a compact connected group, equipped with a normalized Haar measure $\mu_{\mathcal{K}}$, acting continuously on a locally compact second countable space $X$.
		Let $A\subset X$ be a closed subset and $F\subset X$ a compact subset.
		Assume that
		\begin{equation*}
			\mu_{\mathcal{K}}\left\{k\in \mathcal{K} \mid k.x\in A \right\}=0,
		\end{equation*}
		for any $x\in F$.
		Then for any $\varepsilon>0$ there exists an open subset $A\subset A^{\varepsilon}$ so that 
		\begin{equation*}
			\mu_{\mathcal{K}}\left\{k\in \mathcal{K} \mid k.x \notin A^{\varepsilon} \right\} \leq \varepsilon,
		\end{equation*}
		for all $x\in F$.
	\end{lem}
	The lemma follows as any $\mathcal{K}$-representation is \emph{analytic} and the various tubes are \emph{analytic varieties}.
	We will use this lemma in order to show there is no concentration of mass near tubes, considering $\mathcal{K}=M$.
	\begin{cor}[Non-concentration of mass of $M$ orbits on tubes]\label{cor:non-concentration-tubes}
		Let $X(H',x')\subset X$ be an $H'$-tube, as defined in~\eqref{eq:tube-def}. Assume that $x\in X$ is $H'$-generic, then for any $u$-translate of $X(H',x')$, $u.X(H',x')$, and any $\varepsilon>0$, there exists an open subset $u.X(H',x')^{\varepsilon}$ containing $u.X(H',x')$ such that
		\begin{equation*}
			m_{M}\left\{m\in M \mid m.x \in u.X(H',x')^{\varepsilon} \right\}\leq \varepsilon.
		\end{equation*}
	\end{cor}
	\begin{proof}
		As $x$ is $H'$-generic, we have that $x\notin X(H',x')$.
		Moreover, as $u$ commutes with $H'$, $u.x\notin X(H',x')$, or equivalently $x\notin u^{-1}.X(H',x')$.
		As $M$ is a connected compact algebraic group, we must have that
		\begin{equation*}
			m_{M}\left\{m\in M \mid m.x \in u.X(H',x') \right\} =0,
		\end{equation*}
		as if not, there exists a density point $m\in M$ for which $m.x\in u.X(H',x')$, and due to analyticity, we must have that $M.x\subset u.X(H',x')$.
		By the lemma, for every $\varepsilon>0$ there exists a thickening of $u.H(X',x')$, say $u.H(X',x')^{\varepsilon}$ so that
		\begin{equation*}
			m_{M}\left\{m\in M \mid m.x\in u.H(X',x')^{\varepsilon} \right\} \leq \varepsilon.
		\end{equation*}
	\end{proof}
	\begin{proof}[Proof of Theorem~\ref{thm:spherical-qual}]
		Fix $\varepsilon>0$.
		In the case of non-compact space, fix $\rho$ as in Corollary~\ref{cor:sup-rho}, then take $\epsilon$ such that the right hand side of~\eqref{eq:non-div-estimate} would be smaller than $\epsilon$. Denote the compact set $\overline{X_{>\epsilon}}$ by $\Omega$. 
        In case of a compact space, simply take $\Omega=X$.
		By Dani-Margulis~\cite[Theorem~$3$]{dm93}, there are finitely-many $H'$-tubes $X(H',x_i)$ such that if $x\in \Omega\setminus \cup_{i=1}^{N}X(H',x_i)$, then $x$ is $H'$-generic.
		For any fixed $R>0$, using Corollary~\ref{cor:non-concentration-tubes} to Lemma~\ref{lem:eskin-margulis-mozes}, applied to each of the translates of the tubes $u_{-R\cdot\hat{v}}.X(H',x_{i})$ and picking $\varepsilon/N$, where $N$ is the number of tubes involved in the subset $K$, we have that if we consider 
		$\Omega'=~\Omega\setminus~\cup_{i=1}^{N}X(H',x_i)^{\varepsilon/N}$, we have that for any $R\gg 0$, 
		\begin{equation}\label{eq:good-set-measure-est}
			m_{M}\left\{ m\in M \mid u_{R\cdot\hat{v}}.m.x\in \Omega' \right\} \geq 1-2\varepsilon.
		\end{equation}
		Using the Dani-Margulis uniformity result, for $R>R_0$ and for any $y\in K'$ we have that 
		\begin{equation}\label{eq:uni-dani-margulis}
			\left\lvert\frac{1}{\vol(B_{R^{\beta}}^{\perp})}\int_{v\in B_{R}^{\perp}}f(u_{v}.y)dv - \int_{X}fd\mu \right\rvert < \varepsilon.    
		\end{equation}
		Therefore one may conclude using~\eqref{eq:good-set-measure-est},~\eqref{eq:uni-dani-margulis}:
		\begin{equation}\label{eq:M-fixed-equi}
			\begin{split}
				&\left\lvert\int_{m\in M}\frac{1}{\vol(B_{R^{\beta}}^{\perp})}\int_{v\in B_{R}^{\perp}}f(u_{v}.u_{R\cdot \hat{v}}.m.x)dvdm_{M} - \int_{X}fd\mu\right\rvert \\
				&\leq \left\lvert\int_{m\in M: u_{R\cdot \hat{v}}.m.x\in K'}\frac{1}{\vol(B_{R^{\beta}}^{\perp})}\int_{v\in B_{R}^{\perp}}f(u_{v}.u_{R\cdot \hat{v}}.m.x)dvdm_{M} - \int_{X}fd\mu\right\rvert + \lVert f \rVert_{\infty}\cdot 2\varepsilon \\
				&\leq \int_{m\in M: u_{R\cdot \hat{v}}.m.x\in K'}\left\lvert \frac{1}{\vol(B_{R^{\beta}}^{\perp})}\int_{v\in B_{R}^{\perp}}f(u_{v}.u_{R\cdot \hat{v}}.m.x)dv - \int_{X}fd\mu\right\rvert dm_{M} + \lVert f \rVert_{\infty}\cdot 2\varepsilon \\
				&\leq \varepsilon+\lVert f \rVert_{\infty}\cdot 2\varepsilon.
			\end{split}
		\end{equation}
		In order to move from~\eqref{eq:M-fixed-equi} to the general equidistribution result, we will restrict our results to functions of fixed $M$-type.
		Using~\eqref{eq:M-type-func} we get that
		\begin{equation*}
			\begin{split}
				&\int_{m\in M}\frac{1}{\vol(B_{R^{\beta}}^{\perp})}\int_{v\in B_{R}^{\perp}}f(m.u_{v}.u_{R\cdot \hat{v}}.m^{-1}.x)dvdm_{M} \\
				&\ \ = \int_{m\in M}\frac{1}{\vol(B_{R^{\beta}}^{\perp})}\int_{v\in B_{R}^{\perp}}\sum_{i=1}^{n}a_{i}(m)f_{i}(u_{v}.u_{R\cdot \hat{v}}.m^{-1}.x)dvdm_{M} \\
				&\ \ = \sum_{i=1}^{n}a_{i}(m)\int_{m\in M}\frac{1}{\vol(B_{R^{\beta}}^{\perp})}\int_{v\in B_{R}^{\perp}}f_{i}(u_{v}.u_{R\cdot \hat{v}}.m^{-1}.x)dvdm_{M}.
			\end{split}
		\end{equation*}
		As the inner integral tends to $0$, we have that for $R\gg R_0$,
		\begin{equation*}
			\left\lvert \int_{m\in M}\frac{1}{\vol(B_{R^{\beta}}^{\perp})}\int_{v\in B_{R}^{\perp}}f(m.u_{v}.u_{R\cdot \hat{v}}.m^{-1}.x)dvdm_{M} \right\rvert \ll_{\sigma,f} \varepsilon.  
		\end{equation*}
	\end{proof}
	
	\begin{rem}
		The above proof works verbatim if instead of integrating over the whole sphere as in~\eqref{eq:M-avg}, one integrates with respect to some absolutely-continuous density defined over $M$, as the non-divergence estimate will work (by choosing maybe a smaller $\varepsilon$) and the the thickening result of Eskin-Margulis-Mozes in Lemma~\ref{lem:eskin-margulis-mozes} will still apply assuming only absolute-continuity of the density due to analyticity of the tubes which are involved in $K'$.
	\end{rem}

	\section{Applications of joinings of Kronecker system and unipotent flow over semi-simple homogeneous spaces}
	
	The proofs in this section follows an idea of A. Venkatesh regarding sparse equidistribuiton appearing in~\cite[Theorem~$3.1$]{venkatesh10} which was developed in greater generality by the author in~\cite{katz}.
	Unipotent flows (on homogeneous spaces of semi-simple Lie groups) are mixing (due to the Howe-Moore theorem) and therefore are disjoint from any Kronecker system, by Furstenberg's theorem.
	In some cases, where quantitative equidistribution result for the unipotent flow is available, one may use this result (together with a quantitative mixing statement) in order to quantify this disjointness statement (see Theorem~\ref{thm:venkatesh}).
	The idea of Venkatesh is to interpret this result as a computation of a ''Fourier coefficient`` of the sampled function, and then one may study the sparser average by analyzing the corresponding spectral expansion of the averaging operator, utilizing the quantitative estimate regarding the Fourier transform of the sampled function.
	As the Fourier coefficients are decaying (in a quantitative fashion), one expects that such a summation process would work for averages which are sparser than regular ball averages encountered in the ergodic theorem.
	In~\cite[Theorem~$3.1$, Equation~(3.7)]{venkatesh10}, Venkatesh used quantitative information regarding sampling along arithmetic progressions (which can be thought of as sparse samples from the regular ergodic average) in order to deduce a sparse equidistribution result about the horocyclic flow.
	
	We are demonstrating this philosophy by introducing other types of sparse averages, in the multidimensional setting and analyzing them by similar techniques.
	
	we remark here that most of our averages are ''ball-like``, in the sense that one may smooth the average in the flow direction with a controllable error.
	This smoothing effect essentially effects the summation average as transforming its kernel for being ''almost $L^{1}$`` (in practice, it truncates the Fourier expansion in our computations).
	Similar phenomenon happens with the ball averages in~\cite[Proof of Lemma~$3.1$]{venkatesh10}, when utilizing the Van-der-Corput trick, one smooths the ball average, and notice that the convolution of two measures supported on balls (in the flow direction) become absolutely-continuous with respect to the Lebesgue measure (over the orbit).
	
	We will use the following auxiliary construction.
	Fix $R>0$ large. Given a smooth function $f:X\to\mathbb{C}$, $x_0\in X$, we define the \emph{model of $f$} of radius $R$,  $F_{R}:\mathbb{R}^{d}\to\mathbb{R}$ to be
	\begin{equation*}
		F_{R}(v)=\begin{cases}
			f(u_{R\cdot v}.x_0) &\lVert v \rVert \leq 1 \\
			0 &\text{otherwise}
		\end{cases},
	\end{equation*}
	
	As this function is a continuous function supported in $L^{2}[-1,1]^{d}$, we may expand $F_{R}$ in a Fourier series
	\begin{equation}\label{eq:Fourier-expansion}
		F_{R}(v) = \sum_{\overline{n}\in \mathbb{Z}^{d}} a_{\overline{n},R}e_{\overline{n}}(v/2),
	\end{equation}
	where we use the following notation
	\begin{equation*}
		e_{\overline{n}}(v)=e^{2\pi i \left<\overline{n},v\right>},
	\end{equation*}
	and the Fourier coefficients $\left\{a_{\overline{n},R} \right\}_{\overline{n}\in\mathbb{Z}^d}$ are defined as
	\begin{equation*}
		\begin{split}
			a_{\overline{n},R} &= \frac{1}{2^{d}}\int_{v\in [-1,1]^{d}} F_{R}(v)\cdot e^{-2\cdot \pi i \left<v,\overline{n}/2\right>}dv,
		\end{split}
	\end{equation*}
	with equality in the $L^2$-sense.
	
	Using the definition of $F_{R}(v)$, we get
	\begin{equation}\label{eq:Fourier-coefficient}
		\begin{split}
			a_{\overline{n},R} &= \frac{1}{2^{d}}\int_{v\in [-1,1]^{d}} F_{R}(v)\cdot e^{-2\cdot \pi i \left<v,\overline{n}/2\right>}dv \\
			&= \frac{1}{(2\cdot R)^{d}}\int_{\lVert v \rVert \leq R}f(u_{v}.x_0)e^{-2\cdot \pi i \left<v/R,\overline{n}/2\right>}dv. 
		\end{split}
	\end{equation}
	
	As $f$ was assumed to be smooth, $F_{R}$ is smooth, hence the Fourier expansion~\eqref{eq:Fourier-expansion} converges pointwise on the interior points of the cube.
	
	We will use the following definition.
	\begin{defn}\label{def:quantitative-equi-rate}
		We say that a point $x_{0}\in G/\Gamma$ has an effective equidistribution theorem of rate $\gamma>0$ with respect to $H$ and a some Sobolev norm of order $K$ if there exists $\gamma>0$, $R_{0}(x_{0})>0$ such that for any function $f:X\to \mathbb{C}$ of compact support and vanishing integral and for any $R>R_0$ 
		\begin{equation*}
			\left\lvert \frac{1}{\vol_{H}(B_{R}^{H})} \int_{h\in B_{R}^{H}} f(h.x_0)dh \right\rvert \ll_{f} R^{-\gamma}.
		\end{equation*}
	\end{defn}
	
	The following quantitative disjointness theorem was essentially achieved by A. Venkatesh, with various adaptions and generalizations
	\begin{thm}[\cite{venkatesh10} Lemma~$3.1$, \cite{ubis2016effective} Lemma~$2.6$, \cite{sarnakubis} Proposition~$3.1$, \cite{katz} Theorem~$1.6$]\label{thm:venkatesh}
		In the previous settings, assume that $f:X\to\mathbb{R}$ is a smooth function with finite Sobolev norm of order $K$ and vanishing integral and that $x_0\in X$ is an $H$-generic point with effective equidistribution theorem of rate $\gamma>0$, then there exists $\gamma'>0$ such that
		\begin{equation*}
			\left\lvert  \frac{1}{R^{d}}\int_{\lVert v \rVert\leq R}f(u_{v}.x_0)e^{2\pi i \left<v,z\right>}dv \right\rvert \ll_{f} R^{-\gamma'},
		\end{equation*}
		for any $z\in \hat{\mathbb{R}^{d}}$.
	\end{thm}
	
	Using the quantitative disjointness theorem, we may effectively bound the Fourier coefficient of the model of $f$ as follows
	\begin{prop}[Disjointness bound]
		Assume that the condition of Theorem~\ref{thm:venkatesh} holds then for any $R>1$ we have the following bound for the Fourier coefficients:
		\begin{equation*}
			\lvert a_{\overline{n},R} \rvert \ll_{f} R^{-\gamma'}.
		\end{equation*}
	\end{prop}
	The assertion follows immediately from inspection of the explicit expressions of the Fourier coefficients $\left\{a_{\overline{n},R}\right\}_{\overline{n}\in\mathbb{Z}^{d}}$ obtained in ~\eqref{eq:Fourier-coefficient}.

	Fix $\phi$ to be a standard positive mollifier on $\mathbb{R}^d$ and denote $\phi_{\delta}(v)=~\frac{1}{\delta^{d}}\phi(v/\delta)$.
	As our averages are not integrable, we will modify them by convolving the averaging operator with the mollifier to achieve a smoother average.
	
	\begin{prop}[Estimation of smoothing error] In the previous notation, the following estimate holds:
		$$\lvert A_{R,\omega}f(x_0)-(A_{R,\omega}\star \phi_{\delta})f(x_0) \rvert \ll_{f} \frac{\delta}{R^{\omega}}.$$
	\end{prop}
	\begin{proof}
		Using the explicit formula for $A_{R,\omega}$, we infer that this difference is bounded by $\lVert f \rVert_{\infty}\cdot m(\partial A_{R,\omega})^{\delta}$, where $m(\partial A_{R,\omega})^{\delta}$ stands for the measure of a $\delta$-thickening of the boundary of the annulus.
		The natural measure associated with each boundary component of the annulus is $O(R^{d-1})$ and the normalization is done by $O(R^{\omega}\cdot R^{d-1})$, and so the required estimate holds.
	\end{proof}

	So from now on, we will assume that the average is smoothed by $\phi_{\delta}$ at the expanse of the error from the previous proposition.

	\begin{prop}[Truncation effect]\label{prop:truncation}
		In the same settings as above, there exists some $L=L(d)>0$ such that 
		\begin{equation*}
			\left\lvert (A_{R,\omega}\star\phi_{\delta})\left(\sum_{\lVert \overline{n}\rVert \geq \delta^{-L}} a_{\overline{n},R}e_{\overline{n}}(v)\right) \right\rvert \ll_{f} \delta.
		\end{equation*}
	\end{prop}
	\begin{proof}
		By computation we have that
		\begin{equation*}
			\begin{split}
				(A_{R,\omega}\star\phi_{\delta})e_{\overline{n}}(v) &= \widehat{(A_{R,\omega}\star\beta_{\delta})}(\overline{n})\\
				&= \hat{A_{R,\omega}}(\overline{n}) \cdot \hat{\phi_{\delta}}(\overline{n}),
			\end{split}
		\end{equation*}
		by the convolution property of Fourier transform.
		By the non-stationary phase method~\cite[\S8.1 Proposition~$4$]{stein2}, we have that for any $K>0$
		\begin{equation*}
			\lvert \hat{\phi_{\delta}}(\overline{n}) \rvert \ll_{K} (\delta\cdot \lVert \overline{n} \rVert)^{-K}.
		\end{equation*}
		Choosing $K=d+1$ in turn gives the following estimate
		\begin{equation*}
			\begin{split}
				\left\lvert (A_{R,\omega}\star\phi_{\delta})\left(\sum_{\lVert \overline{n}\rVert \geq \delta^{-L}} a_{\overline{n},R}e_{\overline{n}}(v)\right) \right\rvert &\leq \sum_{\lVert \overline{n}\rVert \geq \delta^{-L}} \lvert a_{\overline{n},R}\rvert \left\lvert \hat{A_{R,\omega}}(\overline{n}) \cdot \hat{\phi_{\delta}}(\overline{n}) \right\rvert \\
				&\ll_{f} \sum_{\lVert \overline{n}\rVert \geq \delta^{-L}}(\delta\cdot \lVert \overline{n} \rVert)^{-K} \\
				&\ll_{f} \delta^{L-K},
			\end{split}
		\end{equation*}
		where picking $L=K+1$, proves the claimed estimate.
	\end{proof}

	We recall the following Fourier transforms~\cite[\S3 Theorem~$3.3$]{stein-weiss} 
	\begin{equation}
		\begin{split}
			\hat{\beta_{1}}(z) &= C_{\beta} \frac{J_{d/2}(\lvert z\rvert )}{\lvert z\rvert ^{d/2}} \\
			\hat{\sigma_{1}}(z) &= C_{\sigma} \frac{J_{d/2-1}(\lvert z\rvert )}{\lvert z\rvert ^{d/2-1}},
		\end{split}
	\end{equation}
	where $J_{\star}$ stands for the modified Bessel function, and we have the following renormalization relations:
	\begin{equation*}
		\hat{\beta_{R}}(z) =\hat{\beta_{1}}(R\cdot z) ,\  \hat{\sigma_{R}}(z) = \hat{\sigma_{1}}(R\cdot z).
	\end{equation*}
	Moreover, using the stationary phase method~\cite[\S8.2 Proposition~$5$]{stein2} to evaluate the asymptotics of the Bessel functions we have the following large-mode estimates for the Fourier transforms, for $\lVert z \rVert \geq 1$ \cite[\S8.3 Theorem~$1$]{stein2}:
	\begin{equation*}
		\lvert \hat{\beta_{1}}(z) \rvert \ll \lVert z \rVert^{-(d+1)/2} ,\  \lvert \hat{\sigma_{1}}(z) \rvert \ll \lVert z \rVert^{-(d-1)/2}.
	\end{equation*}
	Using those expansions, we may estimate the Fourier transform of the annulus average as follows:
	\begin{equation*}
		\begin{split}
			\hat{A_{R,\omega}}(z) &= \widehat{\frac{1}{R^{\omega}}\int_{t=0}^{R^{\omega}}\sigma_{R/2+t}dt}(z) \\
			&= \frac{1}{R^{\omega}}\int_{t=0}^{R^{\omega}}\hat{\sigma_1}((R/2+t)(z))dt \\
			&= \frac{C_\sigma}{R^{\omega}}\int_{t=0}^{R^{\omega}} \frac{J_{d/2-1}(\lvert (R+t/2)z\rvert )}{\lvert (R+t/2)z\rvert ^{d/2-1}}dt, \\
		\end{split}
	\end{equation*}
	and in particular
	\begin{equation*}
		\lvert \hat{A_{R,\omega}}(z) \rvert \ll \lVert R\cdot z \rVert^{-\frac{d-1}{2}},
	\end{equation*}
	for $\lVert R\cdot z \rVert \geq 1$.
	
	We are now in position to prove the equidistribution result for annulus.
	
	\begin{proof}[Proof of Theorem~\ref{thm:annuli-equi}]
		We calculate the averaging of $f$ over the annulus via the Fourier series expansion as follows
		\begin{equation}\label{eq:annu-estimate}
			\begin{split}
				\lvert A_{R,\omega}f(x_0) \rvert &= \lvert A_{R,\omega}\star\beta_{\delta}\sum_{\lVert\overline{n}\rVert \geq 0} a_{\overline{n},R}e_{\overline{n}}(v) \rvert +O_{f}(\delta/R^{\omega})\\
				&= \lvert A_{R,\omega}\star\beta_{\delta}\sum_{\lVert\overline{n}\rVert \leq \delta^{-L}} a_{\overline{n},R}e_{\overline{n}}(v) \rvert +O_{f}(\delta)+O_{f}(\delta/R^{\omega}) \\
				&\leq \lvert \sum_{\lVert\overline{n}\rVert \leq \delta^{-L}} a_{\overline{n},R}\hat{A_{R,\omega}}(\overline{n}) \rvert +O_{f}(\delta/R^{\omega}) \\
				&\ll_{f} \sum_{\lVert\overline{n}\rVert \leq \delta^{-L}}\lvert a_{\overline{n},R}\rvert \frac{1}{\lVert \overline{n}\rVert^{(d-1)/2}}  +(\delta/R^{\omega})\\
				&\ll_{f} R^{-\gamma'}\cdot(\delta)^{-L(d+1)/2} +(\delta/R^{\omega}).
			\end{split}
		\end{equation}
		Picking $\delta$ so that
		\begin{equation*}
			\delta^{1+L(d+1)/2}=R^{-\gamma'},
		\end{equation*}
		or equivalently
		\begin{equation*}
			\delta=R^{-\frac{\gamma'}{\frac{d^2+3d+4}{2}}},
		\end{equation*}
		shows that picking any $\omega=R^{1-\frac{2\gamma'}{d^2+3d+4}+\epsilon}$ for any $\epsilon>0$, gives an effective estimate in~\eqref{eq:annu-estimate}, so one may pick
		$\omega_{\text{critical}}=1-\frac{2\gamma'}{d^2+3d+4}$.
	\end{proof}

	\begin{rem}
		Another way to prove the preceding result, in a more geometrical fashion, is to cover the annulus by balls of radius $R^{\omega}$, and deduce the result from individual equidistribution of each such ball.
	\end{rem}
	
	We end this section by showing another application of Venkatesh's estimate, towards a singular Bochner-Riesz mean.
	We recall the following definitions from harmonic analysis.
	\begin{defn}
		The Bochner-Riesz mean of order $\alpha\geq 0$ for a a function $F\in C_{c}(\mathbb{R})$ is defined as
		\begin{equation*}
			BR_{1,\alpha}F(x) = \frac{1}{I^{\alpha}_{1}}\int_{\lVert v \rVert \leq 1} \left(1-\lVert v \rVert^{2} \right)^{\alpha}_{+} \cdot F(x+v)dv.
		\end{equation*}
		
		The related summation kernel is the following~\cite[\S4.4 Theorem~4.15]{stein-weiss}
		\begin{equation*}
			\hat{BR_{1,\alpha}}(z) = \pi^{-\alpha}\Gamma(1+\alpha)\lVert z \rVert^{-n/2-\alpha}\cdot J_{n/2+\alpha}(\lVert z \rVert),
		\end{equation*}
		where $\Gamma$ stands for the Gamma function, and $I_{1}^{\delta}$ is a proper normalization factor.
		We may define via this kernel the Bochner-Riesz mean of order $\alpha$ by 
		\begin{equation*}
			BR_{R,\alpha}F(x) = \frac{1}{I_{1}^{\alpha}}\int_{\lVert v \rVert \leq 1}(1-\lVert v \rVert^{2})^{\alpha}_{+}\cdot F(R\cdot(x+v))dv.
		\end{equation*}

		For $\alpha=0$, the resulting averages are the same as the ball averages. By means of analytic continuation of the Gamma functions and the Bessel function one may deduce that for $\alpha=-1$, the resulting averages are the same of the sphere averages (by considering the associated Fourier multipliers in a distributional sense). Hence varying $\alpha$ between $0$ and $-1$ amounts to interpolating the averaging between the ball average and the sphere average, leading to sparser averages as $\alpha$ goes to $-1$.
	\end{defn}
	In multivariate harmonic analysis, the Bochner-Riesz kernels for $\alpha>~0$ take a prominent role, as the ball averages do not always converge for various $L^{p}$ spaces~\cite[\S4.4]{stein-weiss}
	
	We have the following theorem regarding Bochner-Riesz means of negative order in abelian horospheres
	\begin{thm}\label{thm:Bocher-Riesz}
		Let $X=G/\Gamma$, where $G$ is a linear semisimple Lie group, and $\Gamma$ is a lattice in G.
		Let $H\leq G$ be a horospherical group with respect to an $\mathbb{R}$-diagonalizable element, which we that $H$ is abelian, $H\simeq \mathbb{R}^d$ for $d\geq 2$.
		Then for every $x\in X$ which is $H$-generic, and equidistributes with a rate $\gamma>0$ with respect to functions with finite $K$-Sobolev norm, for any $\alpha>-1$, and any smooth and bounded function $f$ with finite $K$-Sobolev norm we have
		\begin{equation*}
			\left\lvert BR_{R,\alpha}f(x)-\int_{X}fdm \right\rvert \ll_{f} R^{-\gamma'},
		\end{equation*}
		for some $\gamma'=\gamma'(\gamma)>0$.
	\end{thm}
	
	In order to bypass certain convergence issues (as the kernel of the Bochner-Riesz for this range of exponents is not $L^{1}$-integrable), we introduce a smoothing operator $\phi_{\delta}$.
	The following quantifies the effect of the smoothing operator over the Bochner-Riesz mean.
	\begin{prop}[Smoothing effect]
		For $\alpha>-1$, we have the following estimate
		\begin{equation*}
			\left\lvert BR_{R,\alpha}f(x_0)-BR_{R,\alpha}\star\phi_{\delta}f(x_0)\right\rvert \leq O_{f}\left(\delta^{\alpha+1}\right).
		\end{equation*}
	\end{prop}
	\begin{proof}
		Explicitly computing the difference we have
		\begin{equation*}
			\begin{split}
				&\left\lvert BR_{R,\alpha}f(x_0)-BR_{R,\alpha}\star\phi_{\delta}f(x_0)\right\rvert \\
				&\ \ =\bigg\lvert\frac{1}{I_{1}^{\alpha}}\int_{\lVert v \rVert=0}^{1}\big( (1-\lVert v \rVert)^{\alpha}_{+}\cdot F_{R}(v) \\
				&\ \ \ \ \ - \frac{1}{C_{d}\cdot\delta^{d}}\int_{\lVert h\rVert \leq \delta}(1-\lVert v+h \rVert)^{\alpha}_{+}\cdot F_{R}(v+h)\cdot\phi(h/\delta)d\lVert h\rVert\big)  d\lVert v\rVert \bigg\rvert.
			\end{split}
		\end{equation*}
		Therefore the difference is dominated by \begin{equation*}
			\begin{split}
				&\frac{4}{I_{1}^{\alpha}}\int_{1-\delta \leq \lVert v \rVert\leq 1}(1-\lVert v \rVert^{2})^{\alpha}_{+}\cdot \lVert f \rVert_{\infty}\cdot \lVert \phi \rVert_{\infty} d\lVert v\rVert\\ 
				&\leq \frac{8}{I_{1}^{\alpha}}\int_{1-\delta \leq \lVert v \rVert\leq 1}(1-\lVert v \rVert)^{\alpha}_{+}\cdot \lVert f \rVert_{\infty}\cdot \lVert \phi \rVert_{\infty}d\lVert v\rVert \\
				&\ll_{f} \frac{8}{I_{1}^{\alpha}}\int_{\lVert v \rVert \leq \delta}\lVert v \rVert^{\alpha}d\lVert v \rVert \\
				&\ll_{f} \frac{8}{I_{1}^{\alpha}}\cdot \frac{\delta^{1+\alpha}}{1+\alpha} \\
				&\ll_{f} \delta^{1+\alpha}.
			\end{split}
		\end{equation*}
	\end{proof}
	We note that the assumption $\alpha>-1$, which amounts to $\alpha+1>0$, was crucial in the previous computation, as $\alpha+1=0$ would lead to a logarithmic singularity, and in particular, the resulting error \emph{will not} tend to $0$ as $\delta$ approaches $0$.
	
	Using the smoothing operator, we may effectively truncate the Fourier series summation.
	\begin{prop}[Truncation effect]
		In the same settings as above, there exists some $L=L(d)>0$ such that 
		\begin{equation*}
			\left\lvert (B_{R,\alpha}\star\beta_{\delta})\left(\sum_{\lVert \overline{n}\rVert \geq \delta^{-L}} a_{\overline{n},R}e_{\overline{n}}(v)\right) \right\rvert \ll_{f} \delta.
		\end{equation*}
	\end{prop}
	
	\begin{proof}
		The proof follows the lines of Proposition~\ref{prop:truncation}.
		By computation we have that
		\begin{equation*}
			\begin{split}
				(B_{R,\alpha}\star\beta_{\delta})e_{\overline{n}}(v) &= \widehat{(B_{R,\alpha}\star\beta_{\delta})}(\overline{n})\\
				&= \hat{B_{R,\alpha}}(\overline{n}) \cdot \hat{\beta_{\delta}}(\overline{n}),
			\end{split}
		\end{equation*}
		by the convolution property of Fourier transform.
		By the non-stationary phase method, we have that for any $K>0$
		\begin{equation*}
			\lvert \hat{\beta_{\delta}}(\overline{n}) \rvert \ll_{K} (\delta\cdot \lVert \overline{n} \rVert)^{-K}.
		\end{equation*}
		Choosing $K=d+1$ in turn gives the following estimate
		\begin{equation*}
			\begin{split}
				\left\lvert (B_{R,\alpha}\star\beta_{\delta})\left(\sum_{\lVert \overline{n}\rVert \geq \delta^{-L}} a_{\overline{n},R}e_{\overline{n}}(v)\right) \right\rvert &\leq \sum_{\lVert \overline{n}\rVert \geq \delta^{-L}} \lvert a_{\overline{n},R}\rvert \left\lvert \hat{B_{R,\alpha}}(\overline{n}) \cdot \hat{\beta_{\delta}}(\overline{n}) \right\rvert \\
				&\ll_{f} \sum_{\lVert \overline{n}\rVert \geq \delta^{-L}}(\delta\cdot \lVert \overline{n} \rVert)^{-K} \\
				&\ll_{f} \delta^{L-K},
			\end{split}
		\end{equation*}
		where picking $L=K+1$, proves the claimed estimate.
	\end{proof}
	
	\begin{proof}[Proof of Theorem~\ref{thm:Bocher-Riesz}]
		Continuing in a similar fashion to the proof of Theorem~\ref{thm:annuli-equi}, we have the following estimates
		\begin{equation*}
			\begin{split}
				\lvert B_{R,\alpha}f(x_0)\rvert &\ll_{f} \lvert B_{R,\alpha}\star \beta_{\delta}f(x_0) \rvert + \delta^{\alpha+1} \\
				&\ll_{f} \left\lvert \sum_{\overline{n}\in\mathbb{Z}^{d}}a_{\overline{n},R}\cdot  \hat{B_{R,\alpha}}(\overline{n}) \cdot \hat{\beta_{\delta}}(\overline{n}) \right\rvert + \delta^{\alpha+1} \\
				&\ll_{f} \sum_{\lVert\overline{n}\rVert\leq\delta^{-L}}\lvert a_{\overline{n},R}\rvert\cdot  \lvert\hat{B_{R,\alpha}}(\overline{n})\rvert + \delta+\delta^{\alpha+1} \\
				&\ll_{f} R^{-\gamma'}\cdot \sum_{\lVert\overline{n}\rVert\leq\delta^{-L}} \frac{1}{\lVert\overline{n}\rVert^{(d+1)/2+\alpha}} + \delta^{\alpha+1} \\
				&\ll_{f} R^{-\gamma'}\cdot \delta^{-L\cdot\left(\frac{d-1}{2}-\alpha\right)} + \delta^{\alpha+1}.
			\end{split}
		\end{equation*}
		Optimizing, we have
		\begin{equation*}
			\begin{split}
				\delta &= R^{-\frac{\gamma'}{1+\alpha+L\cdot\left(\frac{d-1}{2}-\alpha\right)}} \\
				&= R^{-\frac{\gamma'}{1+\alpha+(d+2)\cdot\left(\frac{d-1}{2}-\alpha\right)}} \\
				&= R^{-\frac{\gamma'}{(d+1)(d-\alpha)}},
			\end{split}
		\end{equation*}
		which leads to the following decay estimate
		\begin{equation*}
			\lvert B_{R,\alpha}f(x_0)\rvert \ll_{f} R^{-\frac{\gamma'(1+\alpha)}{(d+1)(d-\alpha)}}.
		\end{equation*}
	\end{proof}

	\begin{rem}
		In some very specialized cases, one may actually get $\alpha=-1$ (namely quantitative equidistribution of spheres), but those results seems to make usage of highly specialized structures not available in the general case.
		The main example was studied (in a related problem dealing with equidistribution of translates) by Ubis in~\cite{ubis2016effective}. The following is a simplified presentation of the dynamical side of his argument.
		Consider $G=PSL_{2}(\mathbb{R})$, and $\Gamma \leq G$ be a lattice, which for simplicity we will assume to be uniform.
		By Burger's quantification of Furstenberg's unique ergodicity~\cite[Theorem~$2$]{burger1990} result, we have the following equdistribution statement for $f\in C^{4}(G/\Gamma)$ with vanishing integral
		\begin{equation*}
			\left\lvert\frac{1}{T}\int_{t=0}^{T}f(u_{t}.x_0)dt \right\rvert \ll_{f} T^{-s},
		\end{equation*}
		where $s$ is any number smaller than the spectral gap of the $G$ action on $L^{2}_{0}(G/\Gamma)$.
		Define the integer $d$ to be the least integer so that $d\cdot s>1$.
		Consider the space $X=\prod_{i=1}^{d}\left(G/\Gamma\right)$. This space carries an $\mathbb{R}^{d}$-action by unipotent translations as the various unipotent flows $u_{t}$ on each individual factor $G/\Gamma$ act independently on the direct product space $X$.
		We claim that in this setting, we do have equidistribution of the sphere averages.
		A general nice function with vanishing integral in $L^{2}_{0}(X)$, may approximated by the following tensor products $\otimes_{i=1}^{d}f_{i}$, where each $f_{i}$ is a nice function in $L^{2}_{0}(G/\Gamma)$.
		If needed we may subtract the integral of $f_{i}$ (along the factor) from itself, and so we may assume that either $f_{i}$ is a constant function, or a function of vanishing integral.
		In the case where all the components $f_{i}$ are of vanishing integral, writing the ergodic average along the balls in $\mathbb{R}^{d}$ per the components, utilizing Burger's formula via a Fubini argument, we have that
		\begin{equation*}
			\left\lvert\frac{1}{R^{d}}\int_{\lVert v \rVert \leq R}f(u_{v}.x)dv \right\rvert \ll_{f} R^{-d\cdot s}.
		\end{equation*}
		As $d\cdot s>1$, a thickening estimate of the of the sphere (of length say $\delta$), shows that the sphere average is well approximated (up to an error of $O_{f}(\delta)$ coming from a Lipschitz estimate) by $\frac{1}{R^{d-1}}\int_{R\leq \lVert v \rVert \leq R+\delta}f(u_{v}.x)dv$, which itself get bounded effectively, using Burger's quantitative equidistribution theorem by $$2\cdot (R+\delta)^{d-d\cdot s}/R^{d-1} = O\left(R^{1-d\cdot s}\right),$$
		where we suppress multiplicative errors depending on the regularity of $f$.
		As $d\cdot s>1$, this estimate decays effectively.
		In the other case, where at-least one of the $f_{i}$'s is constant, one may project the tensor function $\otimes_{i=1}^{d}f_{i}$ to the product space $\prod_{i}'(G/\Gamma)$, where the prime means we remove the constant factor.
		Note that the spherical measure in $\mathbb{R}^{d}$ gets projected to an absolutely-continuous measure over the ball in $\mathbb{R}^{d-1}$, for which one may use the generalized quantitative horospherical equidistribution theorem (c.f. \cite[Theorem~$2$]{burger1990},\cite[Lemma $9.4$]{venkatesh10},\cite{katz}) over the factor space, in order to conclude effective equidistribution over the factor (and hence over the whole product).
		It seems this technique cannot be generalized immediately to the $SO(d,1)$ case discussed previously, due to the presence of the $M$-group (or more specifically, the non-abelian maximal compact subgroup $K$, which includes the $M$ group as a subgroup, and does not allow an ''individual`` separation of the axis involved).
		Nevertheless, this technique can be generalized into other averages along suitably curved surfaces (of small enough codimension) 
		the projections of those curves would become absolutely-continuous measures over the factor spaces as discussed above.
	\end{rem}

	\section{Joinings over nilmanifolds}
	We remark here briefly about the relationship between joinings and analysis over nilmanifold.
	Given a group $H$ acting on a nilmanifold $N/\Lambda$ by translations, the basic object one is interested in is the $H$ nilsequence defined by a given function $f:N/\Lambda\to\mathbb{R}$ and a basis point $x_0\in N/\Lambda$, which is the sampling sequence of $f$ along the orbit $H.x_0$.
	By fundamental properties of nilflows, one may only consider a specific subclass of functions $f$, so that the $H$-nilsequence is actually an $H$-nilcharacter.
	For a given $H$-nilcharacter $f$, when one tries to study its ergodic average $\int_{h} f(h.x_0)dh$, one may use the Van-der-Corput trick in order to study the closely-related average
	\begin{equation*}
		\int_{h}\int_{h'}f(h'.h.x_0)\cdot\overline{f(h.x_0)}dhdh'.
	\end{equation*}
	By the basic properties of nilcharacters, the ''derived expression`` $$f(h'.h.x_0)\cdot\overline{f(h.x_0)}$$ (which can be thought of as an evaluation of the function $(h'.f) \cdot \overline{f}$ over an $H$-orbit contained in an $H$-joining) may be realized over a subnilmanifolds (of strictly smaller complexity).
	One may repeat this van-der-Corput smoothing trick, each time reducing the complexity of the ''differentiated expression`` until one reaches to classical analysis over the horizontal torus embedded in the nilmanifold, where one may apply  classical techniques from Fourier analysis.

	Let $X=N/\Lambda$, we define the horizontal torus associated to $X$ as $X_{ab}=N/[N,N]\Lambda$. We assume that $\dim X_{ab}=p$ and identify its dual group (the group of horizontal characters) with $\mathbb{Z}^p$.
	L. Green's theorem, have been generalized by Green-Tao in the following effective version~\cite[Theorem~$8.6$]{greentao12}:
	\begin{thm}[Green-Tao]\label{thm:green-tao}
		Let $0<\delta<1/2$ and let $d,R\geq 1$ and a given nilmanifold $N/\Lambda$ equipped with an $\mathbb{R}^{d}$ action by translations.
		Then either for all Lipschitz functions $f:N/\Lambda \to \mathbb{R}$ the following holds:
		\begin{equation*}
			\left\lvert \frac{1}{(2R)^{d}}\sum_{v\in [-R,R]^{d}}f(v.x_0)dv - \int_{N/\Lambda}fdm \right\rvert \ll_{f} \delta,
		\end{equation*}
		or there exists an horizontal character $z\in\widehat{\left(N/\Lambda\right)_{ab}}$, $C_{1}=C_{1}(N,d)$ and $C_{2}=C_{2}(N,d)$ so that $\lVert z \rVert \ll \delta^{-C_{1}}$ and
		\begin{equation*}
			\lvert \left<z,\psi(v.x) \right> \rvert \leq \delta^{-C_{2}}.
		\end{equation*}
	\end{thm}
	We note that Green-Tao achieved a much more general theorem dealing with polynomial sequences, generalizing an equidistribution theorem of Leibman.
	The equidistribution rates of Green-Tao, while not explicit, in principal can be extracted from their paper, and we demonstrate it in a special case later.
	The following definition is analogous to Definition~\ref{def:quantitative-equi-rate}.
	\begin{defn}\label{def:quantitatie-equi-rate-nil}
		We will say that $x_{0}\in N/\Lambda$ has \emph{effective equidistribution of rate $\gamma>0$} if for any smooth function $f$ on $N/\Lambda$ we have
		\begin{equation*}
			\left\lvert \frac{1}{(2R)^{d}}\sum_{v\in [-R,R]^{d}}f(v.x_0)dv - \int_{N/\Lambda}fdm \right\rvert \ll_{f} R^{-\gamma},
		\end{equation*}
		where the implied constant depends on some Sobolev norm of $f$.
	\end{defn}
	\begin{rem}\label{rem:effective-equi}
		In view of the strong subspace theorem and the above-mentioned theorem of Green-Tao, every point in $N/\Lambda$ defined over $\overline{Q}$ satisfy ``almost'' effective equidistribution theorem with a rate (it is only almost in the sense that the quantity $R_{0}$ cannot be effectively computed, due to the ineffective nature of the strong subspace theorem).
		Moreover, using Liouville's theorem, one may extract an effective (but not optimal) equidistribution theorem (c.f. the discussion in \cite[Section $5.2$]{katz2}). 
	\end{rem}
	
	Fix $\delta>0$.
	Using the approximation of the sphere introduced in Proposition~\ref{prop:approx} , which gave us extra averaging on the tangent plane, it is enough to study the following type of average
	\begin{equation}\label{eq:int-rewrite}
		\int_{m\in M}\frac{1}{\vol(B_{R^{\beta}}^{\perp})}\int_{v'\in B_{R^{\beta}}^{\perp}}f(m.u._{v'}.u_{R\cdot \hat{v}}.m^{-1}.x)dv'dm.
	\end{equation}
	where $f:N/\Lambda\to \mathbb{R}$ a Lipschitz function.
	It is enough to show that for most elements $m\in M$ (where most depends quantitatively on $\delta$), the inner integral is bounded by a function of $\delta$.
	
	As our function is assumed to be Lipschitz, by discretizing the average to a $\delta$-grid (at the cost of an error of $O_{f}(\delta)$), one simply needs to consider the averages along the discrete samples $v'\in \mathbb{Z}^{d-1}/\delta$ with $\lVert v' \rVert \leq R^{\beta}$.
	Moreover, one may restrict themselves to the case of averaging over a cube in the tangent space rather than a ball.
	
	We now may estimate the norms of the projections in the following manner:
	\begin{equation*}
		\begin{split}
			\lVert \left<z,\psi(u_{mv'}.u_{R\cdot\Hat{v}}.x) \right>\rVert_{v'} &= \lVert \left<z,\psi(u_{m.v'}) \right>\rVert_{v'} \\
			&= \lVert \sum_{i=1}^{p} z_{i}\cdot \psi(m.v')_{i}\rVert_{v'} \\
			&= \sup_{i=1,\ldots,p} \sup_{\lVert m.v' \rVert =1}\lvert z_{i}\cdot \psi(m.v')_{i}\rvert\cdot R^{\beta}/\delta.
		\end{split}
	\end{equation*}
	Denote $\tau:=\sup_{i=1,\ldots,p} \sup_{\lVert m.v' \rVert =1}\lvert z_{i}\cdot \psi(m.v')_{i}\rvert$.
	Hence we need to show that for most elements $m\in M$, we have that $\tau \geq \delta^{-C_{2}+1}\cdot R^{-\beta}$.
	We first note that by the equidistribution of the whole $H$-orbit, integration in polar coordinates and the homogeneity of the expressions, for almost any $m$, we must have $$\tau=\sup_{i=1,\ldots,p} \sup_{\lVert m.v' \rVert =1}\lvert z_{i}\cdot \psi(m.v')_{i}\rvert \neq 0.$$
	Moreover, for any $z\in \mathbb{Z}^{p}$ we have at-least one $\hat{e_{i}}\in\mathbb{R}^{d}$ for which
	\begin{equation*}
		\lvert\left<z,\psi(\hat{e_i})\right>\rvert \neq 0.
	\end{equation*}
	
	Therefore, for a fixed $z$, we may bound the set of directions $m\in M$ such that
	$\sup_{\lVert m.v' \rVert =1}\lvert z_{i}\cdot \psi(m.v')_{i}\rvert \leq \tau$ by $O\left(\frac{\tau}{\lVert z \rVert}\right)$ by a derivative estimate.
	
	Summing over the various characters $z\in \mathbb{Z}^p$ of weight less than $\delta^{-C_{1}}$, we see that the size of the exceptional set of directions, for which 
	\begin{equation*}
		\lvert\left<z,\psi(\hat{e_i})\right>\rvert \leq \tau
	\end{equation*}
	is bounded by $O(\tau \cdot \delta^{-C_{1}\cdot p})$.
	We denote this set of directions by $\mathcal{B}=~\mathcal{B}(\tau)$, and denote its complement by $\mathcal{A}$.
	
	Combining all those estimates we have
	\begin{equation}\label{eq:effective-equi-nil-estimate}
		\begin{split}
			&\left\lvert\int_{m\in M}\frac{1}{\vol(B_{R^{\beta}}^{\perp})}\int_{v'\in B_{R^{\beta}}^{\perp}}f(m.u._{v'}.u_{R\cdot \hat{v}}.m^{-1}.x)dv'dm\right\rvert \\
			&\ \ \leq \left\lvert\int_{m\in \mathcal{B} }\frac{1}{\vol(B_{R^{\beta}}^{\perp})}\int_{v'\in B_{R^{\beta}}^{\perp}}f(m.u._{v'}.u_{R\cdot \hat{v}}.m^{-1}.x)dv'dm\right\rvert \\
			&\ \ + \int_{m\in \mathcal{A} }\left\lvert\frac{1}{\vol(B_{R^{\beta}}^{\perp})}\int_{v'\in B_{R^{\beta}}^{\perp}}f(m.u._{v'}.u_{R\cdot \hat{v}}.m^{-1}.x)dv'\right\rvert dm \\
			&\ \ \ll_{f} \tau \cdot \delta^{-C_{1}\cdot p} + \delta = \delta^{-C_{2}-C_{1}\cdot p+1}/R^{\beta}+\delta.
		\end{split}
	\end{equation}
	Choosing $\delta=R^{-\frac{\beta}{C_{1}\cdot p +C_{2}}}$, gives an effective bound.
	
	We end by demonstrating an explicit bound for the case of $2$-step nilmanifold. We make usage of better bounds than the bounds achieved throughout the proof, as we know explicit bounds of decay of Fourier transforms of spheres, but we note here that one can extract (slightly weaker) bounds from the linear approximation argument we have demonstrated.
	
	Let $X=N/\Lambda$ be a $2$-step nilmanifold, equipped with an $\mathbb{R}^{d}$ minimal action by translations, such that $\dim X_{ab}=p$.
	One particular example to consider is the analogous case to Ubis~\cite{ubis2016effective}, namely let $\mathbb{H}$ be the  $3$-dimensional Heisenberg group, $H=\mathbb{H}(\mathbb{R})$ be the related real Lie group $\Lambda=\mathbb{H}(Z)\leq H$ be the subgroup of integral points of it and define $3d$-dimensional nilmanifold $X=(H/\Lambda)^{d}$.
	We note that as $H_{ab}$ is a $2$-dimensional torus, $X_{ab}$ is a  $2d$-dimensional torus.
	We can endow each factor with a minimal $\mathbb{R}$-action~\cite[Theorem~$1.6$]{greentao12} given by the flow along the subgroups $u_{i}\leq H$ generated by $u_{i}(t)=\exp\left( t\cdot \begin{pmatrix}
		0 & \alpha_{i} & \gamma_{i} \\
		0 & 0  & \beta_{i} \\
		0 & 0  & 0 
	\end{pmatrix}\right)$ with $1,\alpha_{i},\beta_{i}$ being rationally independent for $i=1,\ldots, d$, and the action is given by
	\begin{equation*}
		u_{i}(t).n=(n_{1},\ldots,u_{i}(t).n_{i},\ldots,n_{d}),
	\end{equation*}
	for all $n\in (H/\Lambda)^{d}$.
	
	In order to show quantitative equidistribution, it is enough to consider equidistribution along nilcharacters defined on $X$.
	Those nilcharacters have two types:
	\begin{itemize}
		\item The nilcharacter is abelian, meaning the nilcharacter factor through $[N,N]$, which in that case we identify the nilcharacter as a character of $X_{ab}$.
		\item The nilcharacter does not factor through $[N,N]$.
	\end{itemize}
	In the case of abelian character, one needs to consider
	\begin{equation*}
		\int_{\Hat{v}\in S^{d-1}}e(\left<z,\psi(R\cdot v'.x_0) \right>)d\Hat{v},
	\end{equation*}
	which is bounded by
	\begin{equation*}
		\hat{\sigma}(R\cdot \tau),
	\end{equation*}
	where $\sigma$ is the unit mass measure defined over $S^{d-1}$, and we define $\tau:=\max_{i=1,\ldots,p} \lvert z_i \cdot \psi(v'.x_0)_{i} \rvert$.
	In this case, due to the decay of the Fourier transform, we get a decay rate of $O\left((R\cdot \tau)^{-(d+1)/2} \right)$.
	
	In the other case to consider, one is given a rank $2$ nilcharacter.
	As described in~\cite[\S5]{greentao12}, using the van-der-Corput trick (which results in a square root loss in the decay rate), one needs to study the various differentiated terms
	\begin{equation*}
		f(\psi(h.v'.x_0))\cdot f(\psi(v'.x_0)),
	\end{equation*}
	which in-turn correlate with linear flow over the embedded torus, leading to using the same estimate we achieved above (with maybe at the cost of changing $\tau$ to $\tau\cdot R^{-0.5}$ and having an extra square root due to the van-der-Corput estimate), leading to an estimate of $O\left((R^{0.5}\cdot \tau)^{-(d+1)/4}\right)$.

	\bibliographystyle{plain}
	\bibliography{joining-equidistribution}
	
\end{document}